\documentclass[reqno]{amsart}
\usepackage{latexsym,amssymb,amsthm,amsmath}

\usepackage{amssymb}
\usepackage{xcolor}
\usepackage{amsmath,hyperref}
\hypersetup{colorlinks=true,
	 linkcolor=blue,
 filecolor=magenta,
urlcolor=cyan,}
\theoremstyle{plain}
\newtheorem{theorem}{Theorem}[section]
\newtheorem{lemma}{Lemma}[section]

\numberwithin{equation}{section}

\theoremstyle{remark}
\newtheorem{remark}{Remark}[section]
\allowdisplaybreaks
\begin{document}

\title[Bernstein type estimate for system of equations]
{Bernstein type gradient estimate for system of weighted local heat equations with potential term}

\author[S. Bhattacharyya]{Sujit Bhattacharyya}

\subjclass[2020]{58J05, 58J35}
\keywords{Gradient estimation; Bernstein type estimation; weighted Riemannian manifold; system of equations; Bakry-\'Emery Ricci curvature}

\begin{abstract}
In this article we provide Bernstein type gradient estimates for two system of local weighted heat type equations with potentials on a weighted Riemannian manifold. We derive all possible cases considering linear potential, exponential potential, combining with static manifold and evolving manifold. This work partially resolved the problem raised by Bhattacharyya et al. in \cite{SB-1}.
\end{abstract}
\maketitle
 
\section{Introduction}
In \cite{SB-1}, the authors studied Bernstein type gradient estimate for local weighted heat type equation with potential on static as well as on evolving manifolds. For the evolving case, they considered local Ricci flow and extended local Ricci flow to derive this estimate and showed that the estimate is independent of the Ricci curvature bound. At the end of \cite{SB-1}, they suggested to derive Bernstein type estimate for a system of equations with exponential potential term as a future work. In this article, two system of weighted local heat equations has been considered and four different Bernstein type estimates has been derived under the following cases
\begin{itemize}
    \item on \textbf{static} manifold with \textbf{linear} potential term,
    \item  on \textbf{static} manifold with \textbf{exponential} potential term,
    \item on \textbf{evolving} manifold with \textbf{linear} potential term,
    \item  on \textbf{evolving} manifold with \textbf{exponential} potential term.
\end{itemize}
This partially solved the problem raised in \cite{SB-1}. The first two estimates has been derived in the second section and the later two in the third section. 
%================================================
Heat equations are very much popular in the field of manifold theory because many evolution formulas arrive naturally in the form of a parabolic partial differential equation. Solving those equations in general is a very hard task. Numerical solutions does not gives us the essence of the effect of curvature constraints in those solutions. Since the solution to those partial differential equations depends on the shape of the space, so it is necessary to understand how the curvature shapes the solutions. As a result the method of gradient estimation plays a crucial role. In this technique one can get properties of the solutions of those equations without actually solving them. This technique was initiated in the year 1887 by Harnack \cite{harnack}, which is today known as Harnack type estimation. Later it was further developed by Li-Yau \cite{Li-Yau}. After that Hamilton's introduction to Ricci flow \cite{Hamilton-1} and Perelman's solution to geometrization conjecture in 2002-2004 \cite{Perelman-1,Perelman-2,Perelman-3} this technique becomes globally well known. From this point the theory of gradient estimation becomes a new branch of mathematics, which starts to cover various types of linear and nonlinear partial differential equations in both static and evolving manifolds. In this regard some notable works of Abolarinwa covers elliptic gradient estimates for weighted elliptic equation \cite{Abolarinwa-1} of the form $$\Delta u(x) +au^s(x)+\lambda(x)u(x)= 0,$$ Li-Yau type estimate for weighted parabolic equation \cite{Abolarinwa-2} of the form $$(\partial_t-\Delta+\mathcal{R})u(x,t)=-au(x,t)\log u(x,t),$$ along the generalized geometric flow and finally on nonlinear heat type equation \cite{Abolarinwa-3} of the form $$\begin{cases}
    (\partial_t-\Delta_\phi)u=H(u)\\
    u|_{\partial M}=0
\end{cases}.$$ 
In 2024, Azami \cite{Azami-1} studied the equation $$\partial_t u(x,t)=\Delta u(x,t)-p(x,t)A(u(x,t))-q(x,t)(u(x,t))^{a+1},$$ and derived space-time gradient estimate for its positive solutions. Hui et al. contributed significantly by discovering Liouville type theorems for elliptic equation with weighted Laplacian \cite{Hui-1} and with $p$-Laplacian \cite{Hui-2}. So far we have discussed on the importance of the works on single equation.
%%%%%%%%%%%5 System %%%%%%%%%%%%%%%%%%%%%%%%%%
Now we mention some works on on system of equations. We start with the work of Bhattacharyya et al. \cite{SB-2} on Hamilton and Souplet-Zhang type gradient estimation on evolving manifold for the semilinear parabolic system $$\begin{cases}
    (\Delta_\phi -\partial_t)f = -e^{\lambda_1t}h^p\\
    (\Delta_\phi -\partial_t)h = -e^{\lambda_2t}f^q.
\end{cases}$$ One can also see the works of Winkler \cite{Winkler-1} on Keller-Segal system
$$\begin{cases}
    u_t = \Delta u - \nabla \cdot (u\nabla v),\\
    v_t = \Delta v - v+u.
\end{cases}$$ Also the work of Wu and Yang \cite{Wu-Yang-1} on the system $$\begin{cases}
    (\partial_t - \Delta)u = e^{\alpha t}v^p\\
    (\partial_t - \Delta)v = e^{\beta t}u^q
\end{cases}$$
in this regard is worth mentioning. As mentioned earlier, our previous work \cite{SB-1} is based on Bernstein type estimate for local weighted heat equation $$(\partial_t-\chi^2\Delta_f)u=0,$$ on static manifold and manifold evolving along local Ricci flow $$\partial_t g = -2\chi^2 Ric,$$ and extended local Ricci flow $$\partial_t g = -2\chi^2 Ric+2\alpha\nabla\phi\otimes\nabla\phi.$$ Here $\chi \in C^\infty(M)$ is a real valued smooth function on $M$ with compact support in a smooth bounded domain $\Omega\subset M$. We use this definition of $\chi$ for the rest of the paper. From this work the main motivation for studying Bernstein type estimate comes into play. Now we mention some physical scenarios where these type of local equations plays a significant role.
	\subsection{Physical Meaning}
	The coefficient $\chi^2$ has units of $\text{length}^2/\text{time}$ and determines the rate of diffusion:
	\begin{itemize}
		\item Larger $\chi^2$ means faster spreading and smoothing of disturbances.
		\item Smaller $\chi^2$ means slower diffusion i.e., gradients persist longer.
	\end{itemize}
	
	\subsection*{Thermal Case}
	In heat conduction, \cite{Salazar}
	$$
	\chi^2 = \alpha = \frac{k}{\rho c_p},
	$$
	\noindent where $k$ is thermal conductivity, $\rho$ is density, and $c_p$ is specific heat capacity. This ratio balances the ability to conduct heat versus the capacity to store heat.
	
	\subsection*{Probabilistic Case}
	For Brownian motion, \cite{Grimmett}
	$$
	\text{Var}(B_t) = 2\chi^2 t,
	$$
	\noindent here $\chi^2$ is half the variance growth rate of the random walk.
	
	\subsection*{Finance Case}
	In the Black--Scholes equation \cite[Chapter~15]{Hull}, volatility $\sigma$ plays the role of diffusivity. After a change of variables, the partial differential equation reduces to a heat equation with
	$$
	\chi^2 = \frac{\sigma^2}{2}.
	$$
%======================System of equations ======================
\subsection*{Neural Fields (Biophysics)}
	Equations of the form\cite{Bressloff}
    $$
	(\partial_t - D \Delta)u = v, \quad (\partial_t - D \Delta)v = u
	$$	
	Used in excitatory--inhibitory neural field models.  
    
	\subsection*{Quantum Field Theory.}
	This type of equation \cite{Hatano}
    $$
	(\partial_t - \tfrac{\hbar^2}{2m}\Delta)u = v, \quad (\partial_t - \tfrac{\hbar^2}{2m}\Delta)v = u
	$$
	used to models coupled Bosonic fields with diffusion-like propagation.  
	\(\chi^2 = \hbar^2 / 2m\). 
	
	\subsection*{Image Processing.}
	$$
	(\partial_t - \chi^2 \Delta)I_1 = I_2, \quad (\partial_t - \chi^2 \Delta)I_2 = I_1
	$$
	The above equations \cite{Perona} can be found for Dual-channel diffusion for PDE-based image enhancement.  
	\(\chi^2 = \sigma^2/2\).  

    \subsection*{Chemical Kinetics.}
	For parabolic system with exponential potential system one can see \cite{Roussel} where 
    $$
	\partial_t C_1 - D \Delta C_1 = k_1 e^{C_2}, \quad \partial_t C_2 - D \Delta C_2 = k_2 e^{C_1}
	$$
	provides a Reaction--diffusion systems with exponential growth rates for coupled systems. 
%================================================
Thus it is worth to study system of equations along the well constructed theory of gradient estimation technique and so we are considering a weighted Riemannian manifold $(M^n,g,e^{-f}d\mu)$. Let $\chi:M\to \mathbb{R}$ be a smooth map. Consider the system of parabolic equations with potential term of degree one, i.e., with linear potential,
\begin{equation}\label{eq_heat_1}
\begin{cases}
    \partial_t u = \chi^2 \Delta_f u - v,\\
    \partial_t v = \chi^2 \Delta_f v - u,\\
    u(x,0) = u_0,\ \forall x\in M,\\
    v(x,0) = v_0,\ \forall x\in M.\\
\end{cases}    
\end{equation}
It can be seen that $u(x,t)=u_0(x)$ and $v(x,t) = v_0$, for all $x\in M\setminus supp\{\chi\}$. We first assume that the manifold $M$ is not evolving along any geometric flow i.e., the manifold is a static weighted Riemannian manifold. Next we will be following the methodology as showed in \cite{ChowAA}, and from there we find that 
\begin{eqnarray}\label{eq_delfu2}
    \nonumber \Delta_f u^2 &=& \nabla_i(\nabla_i u^2)-\langle\nabla f,\nabla u^2\rangle\\
    \nonumber &=& \nabla_i (2u \nabla_iu)-2u\langle\nabla f,\nabla u\rangle\\
    \nonumber &=& 2|\nabla u|^2+2u\Delta u -2u\langle\nabla f,\nabla u\rangle\\
    &=& 2|\nabla u|^2+2u\Delta_f u
\end{eqnarray}
\begin{remark}
    Note that our equation is symmetric with respect to $u$ and $v$. This means that if we interchange $u$ and $v$ in the first equation we get the second one, as a result all we need to do is to find evolution equations for one equation and interchanging $u$ and $v$ will lead to the evolution equation for the other.
\end{remark}
After that we consider the manifold $M$ to be evolving along the local Ricci flow and derive the gradient estimate for the equation \eqref{eq_heat_1}. During the calculation it was observed that we can actually solve the problem proposed in \cite{SB-1}, with suitable application of Young's inequality. As a result we investigated the system with exponential potential term given by
\begin{equation}\label{eq_heat_2}
\begin{cases}
    \partial_t u = \chi^2 \Delta_f u + ae^v,\\
    \partial_t v = \chi^2 \Delta_f v + be^u,\\
    u(x,0) = u_0,\ \forall x\in M,\\
    v(x,0) = v_0,\ \forall x\in M,\\
\end{cases}    
\end{equation}
and found that the Bernstein type estimate can be derived for the case $a<0,b<0$. So we have not completely solved the problem but we got a possible way to resolve the issue. However, we will again mention this issue in the last section of this manuscript. We derived Bernstein type estimate for the equation \eqref{eq_heat_2} on both static and evolving manifolds.
%============================================================
\section{Bernstein type estimate on static manifold}
\subsection{Bernstein type gradient estimate for system of weighted local heat equation with linear potential term on static manifold}
\begin{theorem}\label{thm_pol}
    If $u,v\ge 0$ be a solution to the system \eqref{eq_heat_1} on $M$ with $Ric_f^{m-n}\ge -Kg$, then there exist constants $$B_1 = (\Phi_0(K)T+\frac12)\underset{\Omega}{\max}\ u_0+T\underset{\Omega}{\max}\ v_0$$ and $$B_2 = (\Phi_0(K)T+\frac12)\underset{\Omega}{\max}\ v_0+T\underset{\Omega}{\max}\ u_0,$$ such that 
    \begin{equation}\label{eq_estm}
        \begin{cases}
            |\nabla u|^2 \le \frac{B_1}{\chi^2 t},\\
            |\nabla v|^2 \le \frac{B_2}{\chi^2 t},\text{both on $\Omega\times(0,T]$,}
        \end{cases}
    \end{equation}where $$\Phi(\chi,f,m,K)=8m|\nabla\chi|^2-\chi\Delta_f\chi+7|\nabla\chi|^2+\frac{1}{4}+K,$$ and $\Phi_0(K):=\underset{\Omega}{\max}\ \Phi(\chi,f,m,K)>-(1+\frac{1}{2T}),\forall K\ge 0$.
\end{theorem}
\begin{proof}
    Before starting the proof it should be mentioned that the function $\Phi$ will be used to group terms that contains the arguments $(\chi,f,m,K)$, their combinations or their derivatives. Observe that
    \begin{eqnarray}\label{eq_deltu2}
        \nonumber \partial_t (u^2) &=& 2u u_t\\
        \nonumber &=& 2u(\chi^2\Delta_f u-v)\\
        &=& \chi^2(\Delta_f (u^2)-2|\nabla u|^2)-2uv
    \end{eqnarray}
    Next we simplify the following expression
    \begin{equation}\label{eq_f}
        (\partial_t - \chi^2 \Delta_f)(\chi^2 t |\nabla u|^2).
    \end{equation}
    Since $\chi$ is a function on the spatial variable so it is independent of time and thus
    \begin{equation}\label{eq_f_2}
        (\partial_t - \chi^2 \Delta_f)(\chi^2 t |\nabla u|^2)= t(\partial_t - \chi^2 \Delta_f)(\chi^2 |\nabla u|^2)+\chi^2 |\nabla u|^2.
    \end{equation}
    Hence all it remains to get the evolution equation for $\chi^2 |\nabla u|^2$. To derive this we proceed as follows.
    
    Since $M$ is a static manifold so
    \begin{eqnarray}\label{eq_evol1}
        \nonumber\partial_t(\chi^2 |\nabla u|^2) &=& 2\chi^2\langle \nabla u,\nabla u_t\rangle\\
        \nonumber&=& 2\chi^2 \langle \nabla u, \nabla(\chi^2\Delta_f (u) - v) \rangle\\
        &=& 4\chi^3\Delta_f u\langle\nabla u, \nabla \chi \rangle+2\chi^4\langle\nabla u, \nabla\Delta_fu\rangle-2\chi^2\langle\nabla u, \nabla v\rangle
    \end{eqnarray}
    Applying weighted Bochner formula in \eqref{eq_evol1} we get
    \begin{eqnarray}\label{eq_evol_1}
        \nonumber\partial_t(\chi^2 |\nabla u|^2) &=& \chi^4(\Delta_f|\nabla u|^2-2|\text{Hess }u|^2-2Ric_f(\nabla u,\nabla u))\\
        &&+4\chi^3\Delta_f u\langle \nabla u,\nabla \chi \rangle-2\chi^2\langle\nabla u,\nabla v\rangle
    \end{eqnarray}
    and
    \begin{eqnarray}\label{eq_evol_2}
        \nonumber\chi^2\Delta_f(\chi^2 |\nabla u|^2) &=& \chi^2 \left( 2|\nabla u|^2 (\chi\Delta_f\chi + |\nabla \chi|^2) + \chi^2 \Delta_f|\nabla u|^2 + 8\chi\text{Hess}u(\nabla u,\nabla \chi)\right)
    \end{eqnarray}
    From \eqref{eq_evol_2} and \eqref{eq_evol_1} we get
    \begin{eqnarray}\label{eq_main_1}
        \nonumber(\partial_t-\chi^2\Delta_f)(\chi^2|\nabla u|^2) &=& -2\chi^4|\text{Hess }u|^2-2\chi^4 Ric_f(\nabla u,\nabla u) +4\chi^3\Delta_fu\langle\nabla u,\nabla \chi\rangle \\
        \nonumber &&-2\chi^2\langle \nabla u,\nabla v \rangle -2\chi^2|\nabla u|^2(\chi \Delta_f\chi + |\nabla \chi|^2) \\
        &&-8\chi^3 \text{Hess }u(\nabla u,\nabla\chi).
    \end{eqnarray}
    Now we estimate the terms on the right hand side one by one. For $m>n$ we have from \cite{ChowAA}
    \begin{eqnarray}\label{eq_lap_bound}
        \nonumber4\chi^3\Delta_fu\langle \nabla u,\nabla \chi\rangle &\le& 16m\chi^2|\nabla u|^2|\nabla \chi|^2+\chi^4|\text{Hess }u|^2\\
        &&+\frac{2\chi^4}{m-n}\langle\nabla f,\nabla u\rangle^2.
    \end{eqnarray}
    and 
    \begin{eqnarray}\label{eq_hess_bound}
        -8\chi^2\text{Hess }u(\nabla\chi,\nabla u) &\le& 16\chi^2|\nabla u|^2|\nabla \chi|^2+\chi^4|\text{Hess }u|^2.
    \end{eqnarray}
    Here we have used $\frac1m (\Delta_f u)^2\le |\text{Hess u}|^2+\frac{1}{m-n}\langle\nabla f,\nabla u\rangle^2$ from \cite{Hui-1}. Using \eqref{eq_hess_bound}, \eqref{eq_lap_bound} in \eqref{eq_main_1} and applying Cauchy-Schwarz inequality, we infer
    \begin{eqnarray}
        \nonumber(\partial_t-\chi^2\Delta_f)(\chi^2|\nabla u|^2) &\le& 2\chi^2|\nabla u|^2\{8m|\nabla\chi|^2-(\chi\Delta_f\chi+|\nabla\chi|^2)+8|\nabla\chi|^2\}\\
        \nonumber&&+2\chi^2|\nabla u||\nabla v|-2\chi^2Ric_f^{m-n}(\nabla u,\nabla u)\\
        \nonumber &\le& 2\chi^2|\nabla u|^2\Phi(\chi,f,m,K)+\chi^2\left(\frac{|\nabla u|^2}{2}+2|\nabla v|^2\right)\\
        &&-2\chi^2Ric_f^{m-n}(\nabla u,\nabla u)
    \end{eqnarray}
    Here we have used Young's inequality $2ab\le \frac{a^2}{2}+2b^2$. After that we apply the curvature condition $Ric_f^{m-n}\ge -Kg$ and update the definition of $\Phi(\chi,f,m,K)$ by absorbing the term $\frac{1}{2}$ and $K$ into $\Phi(\chi,f,m,K)$, thus we get 
    \begin{eqnarray}
        (\partial_t - \chi^2 \Delta_f)(\chi^2|\nabla u|^2)\le 2\chi^2|\nabla u|^2\Phi(\chi,f,m,K)+2\chi^2|\nabla v|^2.
    \end{eqnarray}
    Since $\Phi_0(K)=\underset{\chi,m,f,\Omega}{max}\Phi(\chi,f,m,K)$ so
    \begin{eqnarray}
        \nonumber (\partial_t - \chi^2\Delta_f)(\chi^2 t|\nabla u|^2) &=& t(\partial_t - \chi^2\Delta_f)(\chi^2|\nabla u|^2)+\chi^2|\nabla u|^2\\
        \nonumber &\le& 2t\chi^2|\nabla u|^2\Phi_0(K)+\chi^2|\nabla u|^2+2t\chi^2|\nabla v|^2\\
        \nonumber &=& 2\chi^2|\nabla u|^2(\Phi_0(K)t+\frac12)+2t\chi^2|\nabla v|^2.
        \end{eqnarray}
        Thus by \eqref{eq_deltu2}, we get
        \begin{eqnarray}
        \nonumber \implies (\partial_t - \chi^2\Delta_f)\left(\chi^2 t|\nabla u|^2+(\Phi_0(K)T+\frac12)u^2+Tv^2\right) &\le& -2\left(\Phi_0(K)T+T+\frac12 \right)uv\\
        &\le& 0,\ \text{as $u,v\ge 0$}.
    \end{eqnarray}
    As $\chi^2t|\nabla u|^2|_{(\Omega\times {0})\cup(\partial\Omega\times [0,T])}\equiv 0$, $u|_{\partial \Omega\times [0,T]}=u_0$ and our equation is symmetric with respect to $u$ and $v$, hence by maximum principle we find that 
    $$\chi^2 t|\nabla u|^2\le \chi^2 t|\nabla u|^2+(\Phi_0(K)T+\frac12)u^2+Tv^2\le (\Phi_0(K)T+\frac12)\underset{\Omega}{\max}\ u_0+T\underset{\Omega}{\max}\ v_0.$$
    $$\chi^2 t|\nabla v|^2\le \chi^2 t|\nabla v|^2+(\Phi_0(K)T+\frac12)v^2+Tu^2\le (\Phi_0(K)T+\frac12)\underset{\Omega}{\max}\ v_0+T\underset{\Omega}{\max}\ u_0.$$
    Assuming 
    $$B_1 = (\Phi_0(K)T+\frac12)\underset{\Omega}{\max}\ u_0+T\underset{\Omega}{\max}\ v_0$$ and $$B_2 = (\Phi_0(K)T+\frac12)\underset{\Omega}{\max}\ v_0+T\underset{\Omega}{\max}\ u_0,$$
    gives the gradient estimates \eqref{eq_estm}.
    \end{proof}
    %=======================================================================================================
    
\subsection{Bernstein type gradient estimate for system of weighted local heat equation with exponential potential term on static manifold}
In this section a possible solution is being provide to resolve the question raised in \cite{SB-1}. It is a possible way because in this section we are going to derive the Bernstein type estimate with the assumption $a<0,\ b<0$.
\begin{theorem}\label{thm_sys_flow}
    If $u,v\ge 0$ be a bounded solution to the system \eqref{eq_heat_2} on $M$ with $Ric_f^{m-n}\ge -Kg$, $u\le \ln b_1,\ v\le \ln b_2$, for some real constants $b_1,b_2>1$, then there exist constants $$C_1 = (\Psi_0(K,a)T+\frac12)\underset{\Omega}{\max}\ u_0+Tb_2^2\underset{\Omega}{\max}\ v_0$$ and $$C_2 = (\Psi_0(K,b)T+\frac12)\underset{\Omega}{\max}\ v_0+Tb_1^2\underset{\Omega}{\max}\ u_0,$$ such that 
    \begin{equation}\label{eq_estm_2}
        \begin{cases}
            |\nabla u|^2 \le \frac{C_1}{\chi^2 t},\\
            |\nabla v|^2 \le \frac{C_2}{\chi^2 t},\text{both on $\Omega\times(0,T]$,}
        \end{cases}
    \end{equation}where $$\Psi(\chi,f,m,K,\xi)=8m|\nabla\chi|^2-\chi\Delta_f\chi+7|\nabla\chi|^2+\frac{\xi^2}{4}+K,$$ and $\Psi_0(K,\xi):=\underset{\chi,m,f,\Omega}{\max}\ \Psi(\chi,f,m,K,\xi)>-\frac{1}{2T},\forall K\ge 0, \xi=a,b$.
\end{theorem}
\begin{proof}
    We follow the same methodology as the previous proof. To avoid redundancy we will recall the equations required for the proof from the earlier result and only highlight the important parts.

    From \eqref{eq_deltu2} we analogously derive
    \begin{eqnarray}\label{eq_deltu2_2}
        \partial_t(u^2) &=& 2\chi^2 \Delta_f(u^2)-2\chi^2|\nabla u|^2+2aue^v,
    \end{eqnarray}
    and we also take \eqref{eq_delfu2} for finding the evolution equation.

    Next we simplify 
    \begin{equation}\label{eq_simpl}
    (\partial_t-\chi^2\Delta_f)(\chi^2 t |\nabla u|^2)=t(\partial_t-\chi^2 \Delta_f)(\chi^2|\nabla u|^2)+\chi^2|\nabla u|^2.
    \end{equation}
    For the first term, using Bochner formula we get
    \begin{eqnarray}
        \nonumber\partial_t(\chi^2|\nabla u|^2) &=& 4\chi^3\Delta_f u\langle\nabla u,\nabla \chi\rangle+\chi^4\{\Delta_f|\nabla u|^2-2|\text{Hess }u|^2\\
        &&-2Ric_f(\nabla u, \nabla u)\}+2ae^v\chi^2\langle\nabla u,\nabla v\rangle.
    \end{eqnarray}
    Thus the evolution equation of $\chi^2|\nabla u|^2$ becomes
    \begin{eqnarray}
        \nonumber(\partial_t-\chi^2\Delta_f)(\chi^2|\nabla u|^2) &=& -2\chi^4|\text{Hess }u|^2-2\chi^4Ric_f(\nabla u,\nabla u) + 4\chi^3\Delta_f\langle \nabla u,\nabla \chi \rangle\\
        \nonumber &&+ 2a\chi^2e^v\langle \nabla u,\nabla v \rangle -2\chi^2|\nabla u|^2\Delta_f\chi^2 - 8\text{Hess }u(\nabla u,\nabla \chi).\\
    \end{eqnarray}
    Applying \eqref{eq_lap_bound}, \eqref{eq_hess_bound} and the Ricci curvature bound, $Ric_f^{m-n}\ge -Kg$, on the above equation we find that 
    \begin{eqnarray}
        \nonumber(\partial_t-\chi^2\Delta_f)(\chi^2|\nabla u|^2) &\le& 2\chi^2|\nabla u|^2 \Psi(\chi,f,m,K) + 2|a|\chi^2e^v|\nabla u||\nabla v|.
    \end{eqnarray}
    Using Young's inequality $2|a|\chi^2 e^v|\nabla u||\nabla v|\le \chi^2\left(\frac{a^2}{2}|\nabla u|^2+2b_2^2 |\nabla v|^2\right)$ and bound of $v$ yields
    \begin{eqnarray}
        (\partial_t-\chi^2\Delta_f)(\chi^2|\nabla u|^2) &\le& 2\chi^2|\nabla u|^2\Psi_0(K,a)+2\chi^2b_2^2|\nabla v|^2.
    \end{eqnarray}
    Here the final definition of $\Psi$ is, $$\Psi(\chi,f,m,K,a)=8m|\nabla \chi|^2-\chi\Delta_f\chi +7|\nabla \chi|^2+\frac{a^2}{4}+K.$$
    In similar way we derive the evolution equation for $\chi^2|\nabla v|^2$ by interchanging $u$ and $v$ in the above calculation.
    \begin{eqnarray}
        (\partial_t-\chi^2\Delta_f)(\chi^2|\nabla v|^2) &\le& 2\chi^2|\nabla v|^2\Psi_0(K,a)+2\chi^2b_1^2|\nabla u|^2,
    \end{eqnarray}
    where $$\Psi(\chi,f,m,K,b)=8m|\nabla \chi|^2-\chi\Delta_f\chi +7|\nabla \chi|^2+\frac{b^2}{4}+K.$$
    Now from \eqref{eq_simpl} we get
    \begin{eqnarray}
        (\partial_t-\chi^2\Delta_f)(t\chi^2|\nabla u|^2) &\le & 2\chi^2|\nabla u|^2(\Psi_0(K,a)T+\frac12)+2Tb_2^2\chi^2|\nabla v|^2.
    \end{eqnarray}
    Since $a<0$, so by previous arguments, we infer
    \begin{eqnarray}
        \nonumber(\partial_t-\chi^2\Delta_f)\left(t\chi^2|\nabla u|^2 + (\Psi_0(K,a)T+\frac12)u^2+Tb_2^2v^2\right) &\le & 2aTb_2^2ve^u+a(2\Psi_0(K,a)T+1)ue^v\\
        &\le& 0
    \end{eqnarray}
    Interchanging $u,v$ and using the fact that $b<0$ we get the other evolution equation
    \begin{eqnarray}
        \nonumber(\partial_t-\chi^2\Delta_f)\left(t\chi^2|\nabla v|^2 + (\Psi_0(K,b)T+\frac12)v^2+Tb_1^2u^2\right) &\le & 2bTb_1^2ue^v+b(2\Psi_0(K,b)T+1)ve^u\\
        &\le& 0
    \end{eqnarray}
    By the same arguments as used in Theorem \ref{thm_pol} we get $$C_1 = (\Psi_0(K,a)T+\frac12)\underset{\Omega}{\max}\ u_0+Tb_2^2\underset{\Omega}{\max}\ v_0$$ and $$C_2 = (\Psi_0(K,b)T+\frac12)\underset{\Omega}{\max}\ v_0+Tb_1^2\underset{\Omega}{\max}\ u_0,$$ which then leads to the estimate \eqref{eq_estm_2}.
\end{proof}
%=========================================================
\section{Bernstein type estimate on manifold evolving along local Ricci flow}
In this section we extend the above results to the case of manifold evolving along local Ricci flow. Henceforth we are denoting $(M^n,g(t),e^{-f}d\mu)$ a manifold which satisfies
\begin{equation}\label{eq_wt_ricci_flow}
    \frac{\partial}{\partial t}g = -2\chi^2Ric.
\end{equation}
This equation is called the local Ricci flow, whose existence has been studied in \cite{ChowAA}. Interested readers can see \cite{SB-1} for additional informations on this flow and its extension. This section has been divided into two parts one with the first system \eqref{eq_heat_1} with polynomial potential function and in the next section we consider the system \eqref{eq_heat_2} with exponential potential function on evolving manifold and derive Bernstein type estimates for each cases. To write the result concisely we need to extend the lemma of \cite{SB-1} to the case of nonzero potential function.
\begin{lemma}[Extension of Lemma 2.2 of \cite{SB-1}]\label{lemma_ext}
    If $u,v$ is a solution to the system of weighted heat type equations \begin{eqnarray}\label{eq_heat_gen}
        \begin{cases}
            (\partial_t - \chi^2\Delta_f)u = \lambda_1(u,v)\\
            (\partial_t - \chi^2\Delta_f)v = \lambda_2(u,v)\\
        \end{cases}
    \end{eqnarray}
    on a weighted Riemannian manifold $(M,g(t),e^{-f}d\mu)$ evolving along the abstract geometric flow \begin{equation}\label{eq_gen_geom_flow}
        \partial_t g_{ij} = -2h_{ij}
    \end{equation}
    where $h_{ij}$ are smooth functions on $M$, then we have the following
    \begin{eqnarray}
        \nonumber (\partial_t - \chi^2\Delta_f)(\chi^2|\nabla u|^2) &=& -2\chi^2(h+\chi^2Ric)(\nabla u,\nabla u)-2\chi^4\text{Hess }f(\nabla u,\nabla u)\\
        \nonumber &&-2\chi^2|\text{Hess }u|^2+4\chi^3\langle\nabla u,\nabla \chi\rangle\Delta_f u -2\chi^2|\nabla u|^2(\chi\Delta_f\chi+|\nabla \chi|^2)\\
        &&-8\chi^3\text{Hess }u(\nabla u,\nabla \chi)-2\chi^2\frac{\partial \lambda_1}{\partial u}|\nabla u|^2- 2\chi^2\frac{\partial \lambda_1}{\partial v}\langle \nabla v, \nabla u \rangle\\
        %--------------
        \nonumber (\partial_t - \chi^2\Delta_f)(\chi^2|\nabla v|^2) &=& -2\chi^2(h+\chi^2Ric)(\nabla v,\nabla v)-2\chi^4\text{Hess }f(\nabla v,\nabla v)\\
        \nonumber &&-2\chi^2|\text{Hess }v|^2+4\chi^3\langle\nabla v,\nabla \chi\rangle\Delta_f v -2\chi^2|\nabla v|^2(\chi\Delta_f\chi+|\nabla \chi|^2)\\
        &&-8\chi^3\text{Hess }v(\nabla v,\nabla \chi)-2\chi^2\frac{\partial \lambda_1}{\partial v}|\nabla v|^2- 2\chi^2\frac{\partial \lambda_1}{\partial u}\langle \nabla u, \nabla v \rangle
    \end{eqnarray}
\end{lemma}
\begin{proof}
To avoid the same calculation, here a sketch of the proof is being provided.\\
We are mainly concerned with the term $\partial_t|\nabla u|^2$ because the other term $\chi^2\Delta_f(\chi^2 |\nabla u|^2)$ is independent of the flow \eqref{eq_gen_geom_flow}. We see that
\begin{eqnarray}
    \nonumber \partial_t |\nabla u|^2 &=& \partial_t (g^{ij}\nabla_iu\nabla_ju)\\
    \nonumber &=&-2h(\nabla u,\nabla u)+2\langle\nabla u,\nabla u_t\rangle\\
    &=&-2h(\nabla u,\nabla u) +\langle\nabla u,\nabla(\chi^2\Delta_fu)\rangle-\langle\nabla u\nabla \lambda_1(u,v)\rangle
\end{eqnarray}
We proceed as in the proof of Theorem \ref{thm_sys_flow} with the extra terms $-2h(\nabla u,\nabla u)$ and $$\langle\nabla u,\nabla \lambda_1(u,v)\rangle=\frac{\partial\lambda_1}{\partial u}\langle\nabla u,\nabla u\rangle+\frac{\partial\lambda_1}{\partial v}\langle\nabla u,\nabla u\rangle,$$ on the right hand side of $(\partial_t-\chi^2\Delta_f)(\chi^2|\nabla u|^2)$. For the term $\chi^2|\nabla v|^2$, just interchange the role of $u,v$ and replace $\lambda_1$ by $\lambda_2$, the desired result will be obtained.
\end{proof}
If the manifold is evolving along the local Ricci flow $\partial_t g= -2\chi^2Ric$ then we simply substitute $h_{ij}=-\chi^2Ric_{ij}$ in Lemma \ref{lemma_ext} and we get
\begin{eqnarray}
\label{eq_evol_ric_ind_1}
        \nonumber (\partial_t - \chi^2\Delta_f)(\chi^2|\nabla u|^2) &=& -2\chi^4\text{Hess }f(\nabla u,\nabla u)-2\chi^2|\text{Hess }u|^2+4\chi^3\langle\nabla u,\nabla \chi\rangle\Delta_f u \\
        \nonumber &&-2\chi^2|\nabla u|^2(\chi\Delta_f\chi+|\nabla \chi|^2)-8\chi^3\text{Hess }u(\nabla u,\nabla \chi)\\
        &&-2\chi^2\frac{\partial \lambda_1}{\partial u}|\nabla u|^2- 2\chi^2\frac{\partial \lambda_1}{\partial v}\langle \nabla v, \nabla u \rangle\\
        %--------------
\label{eq_evol_ric_ind_2}
        \nonumber (\partial_t - \chi^2\Delta_f)(\chi^2|\nabla v|^2) &=& -2\chi^4\text{Hess }f(\nabla v,\nabla v)-2\chi^2|\text{Hess }v|^2+4\chi^3\langle\nabla v,\nabla \chi\rangle\Delta_f v \\
        \nonumber &&-2\chi^2|\nabla v|^2(\chi\Delta_f\chi+|\nabla \chi|^2)-8\chi^3\text{Hess }v(\nabla v,\nabla \chi)\\
        &&-2\chi^2\frac{\partial \lambda_1}{\partial v}|\nabla v|^2- 2\chi^2\frac{\partial \lambda_1}{\partial u}\langle \nabla u, \nabla v \rangle
    \end{eqnarray}
It can be seen from the above two equations that the right hand side is independent of curvature restriction. Hence the resulting estimate will be independent of curvature condition and will depend only on the weight function $f$.
\subsection{Bernstein type gradient estimate for system of weighted local heat equation with linear potential term on evolving manifold}
In this case we have
\begin{eqnarray}\label{eq_lin_pot}
    \nonumber
    \begin{cases}
        \lambda_1(u,v) = v\\
        \lambda_2(u,v) = u
    \end{cases}
    \\
    \implies 
    \begin{cases}
        \frac{\partial \lambda_1}{\partial u} = 0, \frac{\partial \lambda_1}{\partial v} = 1\\
        \frac{\partial \lambda_2}{\partial u} = 1, \frac{\partial \lambda_2}{\partial v} = 0
    \end{cases}
\end{eqnarray}
This leads to our main theorem.
\begin{theorem}
    If $u,v\ge 0$ are solutions to the system \eqref{eq_heat_1} on the weighted manifold $M$ evolving along the local Ricci flow \eqref{eq_wt_ricci_flow} with $|\nabla f|\le K_1$, $\text{Hess} f\ge -K_2g$, then there exist constants $$D_1 = (\Lambda_0T+\frac12)\underset{\Omega}{\max}\ u_0+T\underset{\Omega}{\max}\ v_0$$ and $$D_2 = (\Lambda_0T+\frac12)\underset{\Omega}{\max}\ v_0+T\underset{\Omega}{\max}\ u_0,$$ such that 
    \begin{equation}\label{eq_estm_3}
        \begin{cases}
            |\nabla u|^2 \le \frac{D_1}{\chi^2 t},\\
            |\nabla v|^2 \le \frac{D_2}{\chi^2 t},\text{both on $\Omega\times(0,T]$,}
        \end{cases}
    \end{equation}where $$\Lambda(K_1,K_2\chi,m,n)=K_2\chi^2+8m|\nabla\chi|^2+\frac{\chi^2}{m-n}K_1+8\chi^2-(\chi\Delta\chi+|\nabla\chi|^2)+\frac14 +\chi|\nabla\chi|K_1,$$ and $\Lambda_0(K_1,K_2,m,n):=\underset{\Omega,\chi}{\max}\ \Lambda(K_1,K_2\chi,m,n)>-(1+\frac{1}{2T}),\forall K_1,K_2\ge 0$. 
\end{theorem}
\begin{proof}
    In \eqref{eq_evol_ric_ind_1} we apply \eqref{eq_lin_pot}, \eqref{eq_lap_bound}, \eqref{eq_hess_bound} and the conditions on the weight function $\text{Hess } f\ge -K_2g,\ |\nabla f|\le K_1$, and deduce
    \begin{eqnarray}
        \nonumber(\partial_t-\chi^2\Delta_f)(\chi^2|\nabla u|^2) &\le & 2\chi^2|\nabla u|^2\Lambda_0+2\chi^2|\nabla v|^2\\
        \nonumber\implies (\partial_t-\chi^2\Delta_f)(t\chi^2|\nabla u|^2) &\le & 2\chi^2T|\nabla u|^2\Lambda_0+2\chi^2T|\nabla v|^2\\
        \nonumber &&+\chi^2|\nabla u|^2\\
        \nonumber (\partial_t-\chi^2\Delta_f)(t\chi^2|\nabla u|^2+(\Lambda_0T+\frac12)u^2 + Tv^2) &\le & -2(\Lambda_0T+T+\frac12)uv\\
        \nonumber &\le&0,\text{ as $\Lambda_0\ge -(1+\frac{1}{2T})$}.
    \end{eqnarray}
    With the similar reasoning as mention in the earlier proofs we conclude that $$t\chi^2|\nabla u|^2\le (\Lambda_0T+\frac12)\underset{\Omega}{\max}u_0+T\underset{\Omega}{\max}v_0.$$
    Since our system \eqref{eq_heat_1} is symmetric in $u,v$ so interchanging $u,v$ we get the estimate for $v$
    $$t\chi^2|\nabla v|^2\le (\Lambda_0T+\frac12)\underset{\Omega}{\max}v_0+T\underset{\Omega}{\max}u_0.$$ The proof ends by selecting the constant $D_1,D_2$ as mentioned in the statement.
\end{proof}
%%%%%%%%%%%%%%%%%%%%%%%%%%%%%%%%%%%%%%%%5
%%%%%%%%%%%%%%%%%%%%%%%%%%%%%%%%%%%%%%%%
%%%%%%%%%%%%%%%%%%%%%%%%%%%%%%%%%%%%%%%%%%%
\subsection{Bernstein type gradient estimate for system of weighted local heat equation with exponential potential term on evolving manifold}
Here we have
\begin{eqnarray}\label{eq_exp_pot}
    \nonumber
    \begin{cases}
        \lambda_1(u,v) = ae^v\\
        \lambda_2(u,v) = be^u
    \end{cases}
    \\
    \implies 
    \begin{cases}
        \frac{\partial \lambda_1}{\partial u} = 0, \frac{\partial \lambda_1}{\partial v} = ae^v\\
        \frac{\partial \lambda_2}{\partial u} = be^u, \frac{\partial \lambda_2}{\partial v} = 0
    \end{cases}
\end{eqnarray}
This leads to our main theorem.
Here we mention the last result.
\begin{theorem}\label{thm_sys_flow_2}
    If $u,v\ge 0$ be a bounded solution to the system \eqref{eq_heat_2} on $M$ with $\text{Hess }f\ge -K_1g$, $|\nabla f|\le K_2$, $u\le \ln b_1,\ v\le \ln b_2$, for some real constants $b_1,b_2>1$, $KK_1,K_2\ge 0$, then there exist constants $$E_1 = (\Gamma_0T+\frac12)\underset{\Omega}{\max}\ u_0+Tb_2^2\underset{\Omega}{\max}\ v_0$$ and $$E_2 = (\Gamma_0T+\frac12)\underset{\Omega}{\max}\ v_0+Tb_1^2\underset{\Omega}{\max}\ u_0,$$ such that 
    \begin{equation}\label{eq_estm_2.2}
        \begin{cases}
            |\nabla u|^2 \le \frac{E_1}{\chi^2 t},\\
            |\nabla v|^2 \le \frac{E_2}{\chi^2 t},\text{both on $\Omega\times(0,T]$,}
        \end{cases}
    \end{equation}where $$\Gamma(\chi,K_1,K_2,m,n,\xi)=K_2\chi^2+8m|\nabla\chi|^2+\frac{\chi^2K_1^2}{m-n}-\chi\Delta\chi+7|\nabla\chi|^2+\chi|\nabla\chi|K_1+\frac{\xi^2}{4}+K,$$ and $\Gamma_0(K_1,K_2,m,n,\xi):=\underset{\chi,\Omega}{\max}\ \Gamma(\chi,K_1,K_2,m,n,\xi)>-\frac{1}{2T}, \xi=a,b$.
\end{theorem}
\begin{proof}
    In \eqref{eq_evol_ric_ind_2} we apply \eqref{eq_exp_pot}, \eqref{eq_lap_bound}, \eqref{eq_hess_bound} along with $\text{Hess } f\ge -K_2g,\ |\nabla f|\le K_1$, and find that
    \begin{eqnarray}
        \nonumber (\partial_t-\chi^2\Delta_f)(t\chi^2|\nabla u|^2) \le 2\chi^2T|\nabla u|^2(\Gamma_0T+\frac12) &+& 2\chi^2b_2^2T|\nabla v|^2\\
        \nonumber \implies (\partial_t-\chi^2\Delta_f)(t\chi^2|\nabla u|^2+(\Gamma_0T+\frac12)u^2+b_2Tv^2) &\le & a(\Gamma_0T+\frac12)e^v+bb_2^2Te^u\\
        \nonumber &\le&0,\text{ as $\Gamma_0\ge -\frac{1}{2T}$}.
    \end{eqnarray}
    In similar way as in the earlier proof we conclude that $$t\chi^2|\nabla u|^2\le (\Gamma_0T+\frac12)\underset{\Omega}{\max}u_0+Tb_2^2\underset{\Omega}{\max}v_0.$$
    Again as our system \eqref{eq_heat_2} is symmetric in $u,v$ so interchanging $u,v$ gives the estimate for $v$
    $$t\chi^2|\nabla v|^2\le (\Gamma_0T+\frac12)\underset{\Omega}{\max}v_0+Tb_2^2\underset{\Omega}{\max}u_0.$$ Here also the proof ends by selecting the constant $E_1,E_2$ as mentioned in the statement.
\end{proof}
%=========================================================
\section{Conclusion and Future works}
We realize that studying different gradient estimates can lead to various understanding of the solutions without solving the equations analytically. On manifolds evolving along local Ricci flow, this particular Bernstein type technique allows us to derive an estimate for local heat equation without any curvature restriction. Interested readers are now welcome to find such estimations for Keller-Segal system as mentioned in the introduction. Along with that if one can consider local Keller-Segal system and derive such estimation then it may also lead to some interesting results. Further extending these results to Finsler geometry will also give new insights to this field.

   \vskip6pt
%\vspace{0.1in}\noindent{\bf Acknowledgement:} 
\noindent \textbf{Data availability:} The author declare that all the data generated during the study are in the manuscript and no additional data was generated.\vskip6pt

\noindent \textbf{Declaration of Generative AI and AI-assisted technologies:} The author declare that no generative AI or AI assisted technology is used for the completion of the manuscript.\vskip6pt

\noindent \textbf{Conflict of interest:} The author declare that he has no conflict of interest.\vskip6pt

\noindent\textbf{Acknowledgement:} 
The author acknowledges all the reviewers for giving their valuable suggestions towards the improvement of the results.

%\vspace{.5in}
\vspace{0.1in}
\noindent Sujit Bhattacharyya\\
Department of Mathematics, Siste rNivedita Uniersity, DG 1/2 New Town, Action Area 1, Kolkata - 700156 India\\
Email: \texttt{sujitbhattacharyya.1996@gmail.com}, \texttt{sujit.b@snuniv.ac.in}

\begin{thebibliography}{99}
\bibitem{Azami-1}
S. Azami, \emph{Differential gradient estimates for nonlinear parabolic equations under integral Ricci curvature bounds}, Rev. Real Acad. Cienc. Exactas Fis. Nat. Ser. A-Mat. \textbf{118 (51)} (2024), https://doi.org/10.1007/s13398-024-01552-9

\bibitem{Abolarinwa-1}
A. Abolarinwa, \emph{Gradient estimates for a nonlinear elliptic equation on complete noncompact Riemannian manifold}, Journal of Mathematical Inequalities, \textbf{12 (2)} (2018), 391–402.

\bibitem{Abolarinwa-2}
A. Abolarinwa, N. K. Oladejo and S. O. Salawu, \emph{Gradient estimates for a weighted nonlinear parabolic equation and applications}, Mathematics, \textbf{8 (7)} (2020), 1150–1163.

\bibitem{Abolarinwa-3}
A. Abolarinwa, \emph{Some gradient estimates for nonlinear heat-type equations on smooth metric measure spaces with compact boundary}, arXiv:2309.00763 [math.DG] (2023).
%%%%%%%%%%%%%%%%%%%%%%%%%%%%%%
\bibitem{Bressloff}
P.C. Bressloff, \emph{Spatiotemporal dynamics of neural fields}, J. Phys. A: Math. Theor. \textbf{45} (2012)
%%%%%%%%%%%%%%%%%%%%%%%%%%%%%
\bibitem{ChowAA}
B. Chow, \emph{Ricci flow analytic aspect part II},
AMS.
%===========================================
\bibitem{Salazar}
A. Salazar, \emph{On thermal diffusivity}, Eur. J. Phys., \textbf{24} (2003), 351-358.


\bibitem{SB-1}
S. Bhattacharyya, S. Ghosh and S. K. Hui, \emph{Bernstein type gradient estimation for weighted local heat equation}, Filomat, \textbf{38 (17)} (2024), 6125–6134

\bibitem{SB-2}
S. Bhattacharyya, S. Azami, and S. K. Hui, \emph{Hamilton and Souplet–Zhang type estimations on semilinear parabolic system along geometric flow}, Indian J Pure Appl Math (2024). https://doi.org/10.1007/s13226-024-00586-4
%===========================================
\bibitem{Grimmett}
Grimmett and Stirzaker, \emph{Probability and Random Processes}, Oxford University Press, 3rd ed., 2001.


%=============================================
\bibitem{harnack}
A. Harnack, \emph{Die Grundlagen der Theorie des logarithmischen Potentiales und der eindeutigen Potentialfunktion in der Ebene}, Leipzig: B.G. Teubner, 1887.

\bibitem{Hamilton-1}
R. S. Hamilton, \emph{The Ricci flow on surfaces}, Mathematics and General Relativity (Santa Cruz, CA, 1986), Contemporary Mathematics, \textbf{71} (1988), 237–262.

\bibitem{Hatano}
Hatano and Nelson, \emph{Localization transitions in non-Hermitian quantum mechanics}, Phys. Rev. Lett. \textbf{77} (1996)

\bibitem{Hui-1}
S. K. Hui, A. Abolarinwa and S. Bhattacharyya, \emph{Gradient estimations for nonlinear elliptic equations on weighted Riemannian manifolds}, Lobachevskii Journal of Mathematics, \textbf{44 (4)} (2023), 1332–1340

\bibitem{Hui-2}
S. K. Hui, A. Abolarinwa, S. Bhattacharyya, \emph{Liouville type theorem for weighted p-Laplacian and elliptic gradient estimate}, Sao Paulo Journal of Mathematical Sciences, \textbf{19 (1)} (2025), 1-14.
%=============================
\bibitem{Hull}
J. C. Hull, \emph{Options, Futures, and Other Derivatives}, 10th ed., Pearson, 2017.
%=============================
\bibitem{Li-Yau}
P. Li and S.-T. Yau, \textit{On the parabolic kernel of the Schr\"odinger operator}, Acta Mathematica, \textbf{156 } (1986), 153–201.

\bibitem{Perelman-1}
G. Perelman, \emph{The entropy formula for the Ricci flow and its geometric applications}, arXiv:math/0211159 [math.DG] (2002).

\bibitem{Perelman-2}
G. Perelman, \emph{Ricci flow with surgery on three-manifolds}, arXiv:math/0303109 [math.DG] (2003).

\bibitem{Perelman-3}
G. Perelman, \emph{Finite extinction time for the solutions to the Ricci flow on certain three-manifolds}, arXiv:math/0307245 [math.DG] (2003).

\bibitem{Perona}
Perona and Malik, \emph{Scale-space and edge detection using anisotropic diffusion}, IEEE PAMI (1990).

\bibitem{Roussel}
M.R. Roussel, \emph{Reaction--Diffusion Equations in Chemical Kinetics}, Univ. of Lethbridge (2005).

\bibitem{Winkler-1}
M. Winkler, Finite-time blow-up in the higher-dimensional parabolic-parabolic Keller-Segel system, J. Math. Pure. Appl., 100
(2013), 748-767.

\bibitem{Wu-Yang-1}
H. Wu and X. Yang, Global existence and finite time blow-up for a parabolic system on hyperbolic space, Journal of Mathematical
Physics, 59 (2018), 1-11.

%%%%%%%%%%%%%%%%%%%%%%%%%%%%%%%%%%%%%%
\end{thebibliography}
\end{document}